\documentclass[conference]{IEEEtran}
\IEEEoverridecommandlockouts

\usepackage{cite}
\usepackage{comment}
\usepackage{float}
\usepackage{amsmath,amssymb,amsfonts,amsthm, bbm}
\usepackage{algorithm}
\usepackage{algpseudocode}
\usepackage{graphicx}
\usepackage{subcaption}
\usepackage{textcomp}
\usepackage{xcolor}
\usepackage{tikz}
\usepackage{mathtools}
\usepackage[normalem]{ulem}
\usepackage[margin=0.75in]{geometry}
\usepackage{scalerel}
\usepackage{placeins}
\usepackage{overpic}
\usepackage[colorlinks=true]{hyperref}

\DeclareMathOperator*{\argmax}{\arg\,\max}

\begin{document}

\newtheorem{hypotheses}{Hypotheses}
\newtheorem{problem}{Problem}
\newtheorem{example}{Example}
\newtheorem{definition}{Definition}
\newtheorem{assumption}{Assumption}
\newtheorem{theorem}{Theorem}
\newtheorem{lemma}{Lemma}
\newtheorem{corollary}{Corollary}[theorem]
\newtheorem{proposition}{Proposition}
\newtheorem*{remark}{Remark}
\newtheorem{conjecture}{Conjecture}
\numberwithin{assumption}{section}

\newcommand{\edit}[1]{\textcolor{blue}{#1}}
\newcommand{\ignore}[1]{}
\newcommand{\diff}{\mathop{}\!\mathrm{d}}

\algnewcommand{\Inputs}[1]{%
  \State \textbf{Inputs:}
  \Statex \hspace*{\algorithmicindent}\parbox[t]{.8\linewidth}{\raggedright #1}
}
\algnewcommand{\Initialize}[1]{%
  \State \textbf{Initialize:}
  \Statex \hspace*{\algorithmicindent}\parbox[t]{.8\linewidth}{\raggedright #1}
}

\title{\LARGE \bf
Resilience to Non-Compliance in Coupled Cooperating Systems \vspace{-1.5ex}
}

\author{Brooks A. Butler and Philip E. Par\'{e}$^*$
\thanks{*Brooks A. Butler is with the Department of Electrical Engineering and Computer Science at the University of California, Irvine and~Philip. E. Par\'e is with the Elmore Family School of Electrical and Computer Engineering at Purdue University. Emails: bbutler2@uci.edu and philpare@purdue.edu.
This work was funded by Purdue’s Elmore Center for Uncrewed Aircraft Systems and the National Science Foundation,
grant NSF-ECCS \#2238388.}
}

\maketitle

\begin{abstract}
This letter explores the implementation of a safe control law for systems of dynamically coupled cooperating agents. Under a CBF-based collaborative safety framework, we examine how the maximum safety capability for a given agent, which is computed using a collaborative safety condition,
influences safety requests made to neighbors. We provide conditions under which neighbors may be resilient to non-compliance of neighbors to safety requests, and compute an upper bound for the total amount of non-compliance an agent is resilient to, given its 1-hop neighborhood state and knowledge of the network dynamics. We then illustrate our results via simulation on a networked susceptible-infected-susceptible (SIS) epidemic model.
\end{abstract}

\vspace{-2ex}

\section{Introduction}
Ensuring safety for complex networked systems is a challenging problem across many applications. Some examples of critical networked system applications include smart grid management \cite{tuballa2016review}, uncrewed aerial drone swarms \cite{tahir2019swarms, chung2018survey}, autonomous vehicle networks \cite{plathottam2018next}, multi-agent robot systems \cite{wang2017safety}, and the mitigation of epidemic spreading processes~\cite{pare2020modeling}. While much of the foundational work on safety for dynamical systems via set invariance was performed by Nagumo during the 1940s \cite{blanchini1999set}, the study of safety-critical control has seen a significant resurgence in recent years, due largely to the introduction and refinement of \textit{control barrier functions} (CBFs)~\cite{ames2016control}.
Of particular interest for this work are scenarios where agents in networked systems may be treated as independent entities with individual safety requirements. This dynamically coupled environment motivates the need for collaborative frameworks that can facilitate cooperation between neighbors and enable the collective safety of all agents \cite{lewis2013cooperative,li2017cooperative,wang2017cooperative,yu2017distributed}.

Resilience for networked systems is crucial to achieving reliable performance \cite{ramachandran2021resilient,dobson2019self,filippini2014modeling,reed2009methodology}. In this work, we define resilience similarly to \cite{ramachandran2021resilient}, which is the ability of the network to operate reliably even after experiencing some failures. 
Note that resilience is different from the property of robustness, which is the ability of systems to reject disturbances such as noise.
In our previous work, we developed a CBF-based safety-filter control law that defines a collaborative safety framework for cooperating coupled agents \cite{butler2024collaborativesafety} and applied this framework to a formation control problem with obstacle avoidance \cite{butler2024collabformation}. In this work, we examine the case where, under the previously developed collaborative framework, agents may be unable to fulfill their commitments to maintain their neighbors' safety requirements. Thus, we investigate how resilient each agent is to this possible neighbor failure, which is a first step towards identifying factors that could lead to cascading failures across the network \cite{valdez2020cascading}.
In other words, we investigate the following question: \textit{How resilient is each agent to non-compliant neighbors?}

Thus, in this letter, we make the following contributions:
\begin{itemize}
    \item Under the formulation of our collaborative safety framework \cite{butler2024collaborativesafety}, we show how the computation of the maximum safety capability for a given agent influences its safety requests made to neighbors.
    \item We define a metric for non-compliance and provide an upper bound of the total amount of non-compliance an agent is resilient to given its 1-hop neighborhood state and knowledge of networked dynamics.
\end{itemize}
We introduce notation, networked model definitions, and safety definitions in Section~\ref{sec:preliminaries} and then provide the necessary background for discussing safety in networked systems in the context of our collaborative safety framework from \cite{butler2024collaborativesafety} in Section~\ref{sec:safe_in_net_sys}. We then explore properties of resilience to non-compliance in the context of our collaborative safety framework in Section~\ref{sec:resiliance} and illustrate our results via simulation on a networked susceptible-infected-susceptible (SIS) epidemic model in Section~\ref{sec:simulations}.

\vspace{-1ex}

\section{Preliminaries} \label{sec:preliminaries}
In this section, we define the notation used in this paper, introduce preliminaries for networked dynamic systems, and discuss safety definitions for networked systems.
\subsection{Notation}
\noindent 
Let $\text{Int} \mathcal{C}$, $\partial \mathcal{C}$, $|\mathcal{C}|$ denote the interior, boundary, and cardinality of the set $\mathcal{C}$, respectively. $\mathbb{R}$ and $\mathbb{N}$ are the set of real numbers and positive integers, respectively. Let $D^r$ denote the set of functions $r$-times continuously differentiable in all arguments, and $\mathcal{K}$ the set of class-$\mathcal{K}$ functions. We define $[n] \subset \mathbb{N}$ to be a set of indices $\{1, 2, \dots, n\}$.
We define the Lie derivative of the function $h:\mathbb{R}^N \rightarrow \mathbb{R}$ with respect to the vector field generated by $f:\mathbb{R}^N \rightarrow \mathbb{R}^N$ as
\begin{equation}
    \mathcal{L}_f h(x) = \frac{\partial h(x)}{\partial x} f(x).
\end{equation}
We define higher-order Lie derivatives with respect to the same vector field $f$ with a recursive formula \cite{robenack2008computation}, where $k>1$, as
\begin{equation}
    \mathcal{L}^k_f h(x) = \frac{\partial \mathcal{L}^{k-1}_f h(x)}{\partial x} f(x).
\end{equation}

\subsection{Networked Dynamic System Model} \label{sec:sub:networked_model}
We define a networked system using a graph $\mathcal{G} = (\mathcal{V}, \mathcal{E})$, where $\mathcal{V}$ is the set of $n = \vert \mathcal{V} \vert$ nodes, $ \mathcal{E} \subseteq \mathcal{V}\times \mathcal{V} $ is the set of edges. Let $\mathcal{N}_i^+$ be the set of all neighbors with an incoming connection to node~$i$, where 
\begin{equation}
    \mathcal{N}_i^+ = \{j \in [n]\setminus \{i\}: (i,j) \in \mathcal{E} \}.
\end{equation}
Similarly, all nodes with an outgoing connection from $i$ to $j$ are given by
\begin{equation}
    \mathcal{N}_i^- = \{j \in [n]\setminus \{i\}: (j,i) \in \mathcal{E} \},
\end{equation}
with the complete set of neighboring nodes given by
\begin{equation}
    \mathcal{N}_i = \mathcal{N}_i^+ \cup \mathcal{N}_i^-.
\end{equation}

\noindent Further, we define the state vector for each node as $x_i \in \mathbb{R}^{N_i}$, with $N = \sum_{i \in [n]} N_i$ being the state dimension of the entire system, $N_i^+ = \sum_{j \in \mathcal{N}_i^+} N_j$ the combined dimension of incoming neighbor states, and $x_{\mathcal{N}_i^+} \in \mathbb{R}^{N_i^+}$ denoting the combined state vector of all incoming neighbors.  Then, for each node~$i \in [n]$, we describe its state dynamics, which are potentially nonlinear, time-invariant, and control-affine, as
\begin{equation} \label{eq:sys_dyn_net}
    \dot{x}_i = f_i(x_i, x_{\mathcal{N}_i^+}) + g_i(x_i) u_i,
\end{equation}
where $f_i:\mathbb{R}^{N_i + N_i^+} \rightarrow \mathbb{R}^{N_i}$ and $g_i:  \mathbb{R}^{N_i} \rightarrow \mathbb{R}^{N_i} \times \mathbb{R}^{M_i}$ locally Lipschitz for all $i \in [n]$, and $u_i \in \mathcal{U}_i \subset \mathbb{R}^{M_i}$. 
% Note that in our formulation of a networked dynamic system, we include coupling effects and networked structure as an integral part of our control-free model dynamics $f_i$, rather than treating coupling connections as disturbances or noise. In this sense, we aim to exploit any information the network structure provides in our control design. 
For notational compactness, given a node~$i \in [n]$, we collect the 1-hop neighborhood state as $\mathbf{x}_i = (x_i, x_{\mathcal{N}_i^+})$ and the 2-hop neighborhood state as $\mathbf{x}_i^+ = (x_i, x_{\mathcal{N}_i^+}, x_{\mathcal{N}_j^+}~\forall j \in \mathcal{N}_i^+)$.

\subsection{Safety Definitions} \label{sec:sub:safety_def}
We define a node-level safety constraint for node~$i \in [n]$ with the set
\vspace{-1.5ex}
\begin{equation}\label{eq:safe_set_i_fullstate}
    \begin{aligned}
        \mathcal{C}_i &= \left\{x_i \in \mathbb{R}^{N_i} : h_i(x_i) \geq 0 \right\},
    \end{aligned}
\end{equation}
where $h_i \in D^r, r\geq 1$ and $h_i:\mathbb{R}^{N_i} \rightarrow \mathbb{R}$ is a function whose zero-super-level set defines the region which node $i\in [n]$ considers to be safe (i.e., if $h_i(x_i)<0$, then node~$i$ is no longer safe). 
% We define the safety constraints for the entire networked system as
% \begin{equation}\label{eq:safe_set_net_fullstate}
%     \begin{aligned} 
%         \mathcal{C} &= \mathcal{C}_1 \times \cdots \times \mathcal{C}_n. \;
%     \end{aligned}
% \end{equation}

% \noindent
% Given the definitions of these safety constraints, we define the viable safe regions \cite{breeden2022compositions} for each node as follows.
% \begin{definition}
% A set $\mathcal{S}_i \subseteq \mathcal{C}_i$ is called a \textit{node-level viability domain} for node~$i$ if for every point $x_i(t_0) \in \mathcal{S}_i$ there exist a control signal $u_i(t) \in \mathcal{U}_i$ 
% % and $u_{\mathcal{N}_i^+}(t) \in \mathcal{U}_{\mathcal{N}_i^+}$ 
% with $t \in \mathcal{T}$ such that the trajectory $x_i(\cdot)$ of \eqref{eq:sys_dyn_net} satisfies $x_i(t) \in \mathcal{S}_i$ for all $t \in \mathcal{T}$.
% \end{definition}

\section{Safety in Networked Dynamic Systems} \label{sec:safe_in_net_sys}
 
% \edit{[Intro sentence]}

% \subsection{Safety in Networked Dynamic Systems} \label{sec:safety_in_networks}

One challenge that comes from the definition of node-level safety constraints, as defined in Section~\ref{sec:sub:safety_def}, which depend only on the state of the given node, is that when evaluating the derivative of $h_i(x_i)$ we must incorporate the network influence of its neighbors $j \in \mathcal{N}_i^+$, where we define the derivative of $h_i$ as
\begin{equation*}
    \dot{h}_i(\mathbf{x}_i, u_i) = \mathcal{L}_{f_i} h_i(\mathbf{x}_i) + \mathcal{L}_{g_i} h_i(x_i) u_i. 
    % = \frac{\partial h_i(x_i)}{\partial x}\big( f_i(\mathbf{x}_i) + g_i(x_i) u_i \big).
\end{equation*}
\noindent
Notice that in this case, we can define the node-level constraint function as the mapping $h_i:\mathbb{R}^{N_i} \rightarrow \mathbb{R}$; however, the dynamics of node~$i$, whose dynamics are a function of all neighboring nodes in $\mathcal{N}_i^+$, require that $\dot{h}_i:\mathbb{R}^{N_i+N_i^+} \rightarrow \mathbb{R}$. 

% Therefore, even though we define $h_i$ with respect to only $x_i \in \mathbb{R}^{N_i}$, when computing the Lie derivative we must compute the Jacobian with respect to all network states $\frac{\partial h_i(x_i)}{\partial x}$, which naturally zeros out all states of nodes $j \notin \mathcal{N}_i^+ \cup \{i\}$ when taking only the first derivative, but is relevant when we must consider higher-order derivatives of the constraint function $h_i$. 

Using principles from \textit{high-order barrier functions} \cite{xiao2021high},
we can analyze the effect of the 1-hop neighborhood dynamics on the safety of node~$i$ by computing the second-order derivative of $h_i$, 
% with respect to the networked dynamics defined by \eqref{eq:sys_dyn_net}, 
which can be expressed as

\small
\begin{equation} \label{eq:sec_der_h_i_control_aff}
    \begin{aligned}
        \ddot{h}_i(\mathbf{x}_i^+,u_i, u_{\mathcal{N}_i^+}, \dot{u}_i) &= \sum_{j \in \mathcal{N}_i^+} \big[ \mathcal{L}_{f_j} \mathcal{L}_{f_i}  h_i(\mathbf{x}_i^+) +  \mathcal{L}_{g_j} \mathcal{L}_{f_i}  h_i(\mathbf{x}_i) u_j \big] \\
         &
         % \quad 
         + \mathcal{L}^2_{f_i}  h_i(\mathbf{x}_i) +  u_i^\top \mathcal{L}^2_{g_i} h_i(x_i) u_i + \mathcal{L}_{g_i}h_i(x_i)\dot{u}_i \\ 
         &
         % \quad 
         + \big(\mathcal{L}_{f_i} \mathcal{L}_{g_i}h_i(\mathbf{x}_i)^\top +  \mathcal{L}_{g_i} \mathcal{L}_{f_i}h_i(\mathbf{x}_i)\big)u_i.
    \end{aligned}
\end{equation}
\normalsize

\noindent
Notice that the dynamics of each neighbor $j \in \mathcal{N}_i^+$ 
and $\dot{u}_i$ appear in the second-order differential expression of $h_i$. To assist in our analysis of the high-order dynamics of $h_i$, we make the following assumption.

\begin{assumption} \label{assume:u_dot_func_u}
    For a given node~$i\in [n]$, let $\dot{u}_i := d(u_i)$, where $d(u_i): \mathbb{R}^{M_i} \rightarrow \mathbb{R}^{M_i}$ is locally Lipschitz.
\end{assumption}

While obtaining a closed-form solution for $\dot{u}_i$ may be challenging in some applications, in practice, $d(u_i)$ may be approximated using discrete-time methods. 
Given this structure in the 1-hop neighborhood controlled dynamics, we construct a
series of \textit{high-order control barrier functions}
under Assumption~\ref{assume:u_dot_func_u}, as
\begin{equation} \label{eq:MO_funcs}
    \begin{aligned}
        \psi_i^0(x_i) &:= h_i(x_i) \\
        \psi_i^1(\mathbf{x}_i,u_i) &:= \dot{\psi}_i^0(\mathbf{x}_i,u_i) + \eta_i(\psi_i^0(x_i)) \\
        \psi_i^2(\mathbf{x}_i^+,u_i,u_{\mathcal{N}_i^+}) &:= \dot{\psi}_i^1(\mathbf{x}_i^+,u_i,u_{\mathcal{N}_i^+}) + \kappa_i(\psi_i^1(\mathbf{x}_i, u_i)),
    \end{aligned}
\end{equation}
where $\eta_i,\kappa_i$ are class-$\mathcal{K}$ functions. 
We can also express $\psi_i^2(\mathbf{x}_i^+,u_i,u_{\mathcal{N}_i^+})$ in terms of $h_i$ as
\begin{equation} \label{eq:psi2_terms_h}
    \begin{aligned}
        \psi_i^2(\mathbf{x}_i^+,u_i,u_{\mathcal{N}_i^+}) &= \ddot{h}_i(\mathbf{x}_i^+,u_i,u_{\mathcal{N}_i^+}) + \dot{\eta}_i(h_i(x_i),\mathbf{x}_i,u_i) \\
        & \quad + \kappa_i\big(\dot{h}_i(\mathbf{x}_i,u_i) + \eta_i(h_i(x_i)) \big).
    \end{aligned}
\end{equation}
Further, we may rewrite \eqref{eq:psi2_terms_h} by collecting the terms independent of $u_j$ as
\begin{equation} \label{eq:psi2_grouped_terms}
     \psi_i^2(\mathbf{x}_i^+,u_i,u_{\mathcal{N}_i^+}) = \sum_{j \in \mathcal{N}_i^+} a_{ij}(\mathbf{x}_i) u_j  +  c_i(\mathbf{x}_i^+,u_i),
\end{equation}
where
\begin{equation} \label{eq:a_ij}
    a_{ij}(\mathbf{x}_i) = \mathcal{L}_{g_j} \mathcal{L}_{f_i}  h_i(\mathbf{x}_i)
\end{equation}
and
\footnotesize
\begin{equation} \label{eq:capability_node_i}
    \begin{aligned}
        c_i(\mathbf{x}_i^+,u_i) &= \sum_{j \in \mathcal{N}_i^+} \mathcal{L}_{f_j} \mathcal{L}_{f_i}  h_i(\mathbf{x}_i^+) + \mathcal{L}^2_{f_i} h_i(\mathbf{x}_i) + u_i^\top \mathcal{L}^2_{g_i} h_i(x_i) u_i \\
        & 
        % \quad 
        + \big(\mathcal{L}_{f_i} \mathcal{L}_{g_i}h_i(\mathbf{x}_i)^\top +  \mathcal{L}_{g_i} \mathcal{L}_{f_i}h_i(\mathbf{x}_i)\big)u_i + \mathcal{L}_{g_i}h_i(x_i)d(u_i) \\ 
        & 
        % \quad 
        + \dot{\eta}_i\big(h_i(x_i),\mathbf{x}_i,u_i \big) + \kappa_i\big(\dot{h}_i(\mathbf{x}_i,u_i) + \eta_i(h_i(x_i)) \big).
    \end{aligned}
\end{equation}
\normalsize

\noindent
In this form, we may consider $a_{ij}(\mathbf{x}_i) \in \mathbb{R}^{M_j}$ to be the effect that node $j \in \mathcal{N}_i^+$ has on the safety of node~$i \in [n]$ via its own control inputs $u_j$, and $c_i(\mathbf{x}_i^+,u_i)$ is the effect that node~$i \in [n]$ has on its own safety combined with the uncontrolled system dynamics.
The functions in \eqref{eq:MO_funcs} in turn define the constraint sets
\begin{equation} \label{eq:MO_sets}
    \begin{aligned}
        \mathcal{C}_i^1 &= \{\mathbf{x}_i \in  \mathbb{R}^{N_i+N_i^+}: \psi_i^0(x_i) \geq 0 \} \\
        \mathcal{C}_i^2 &= \{\mathbf{x}_i \in  \mathbb{R}^{N_i+N_i^+}: \exists u_i \in \mathcal{U}_i \text{ s.t. } \psi_i^1(\mathbf{x}_i, u_i) \geq 0 \}.
    \end{aligned}
\end{equation}

% \noindent
We define a \textit{collaborative control barrier function} for node~$i$ that takes into account the control actions of its incoming neighbors $j \in \mathcal{N}_i^+$.

\begin{definition}
    Let $\mathcal{C}_i^1$ and $\mathcal{C}_i^2$ be defined by \eqref{eq:MO_funcs} and \eqref{eq:MO_sets}, under Assumption~\ref{assume:u_dot_func_u}. We have that $h_i$ is a \textbf{collaborative control barrier function} (CCBF) for node~$i \in [n]$ if $h_i \in C^2$ and $\forall \mathbf{x}_i \in \mathcal{C}_i^1 \cap \mathcal{C}_i^2$ there exists $(u_i, u_{\mathcal{N}_i^+}) \in \mathcal{U}_i \times \mathcal{U}_{\mathcal{N}_i^+}$ such that 
    \begin{equation} \label{eq:CCBF_cond}
        \psi_i^2(\mathbf{x}_i^+, u_i, u_{\mathcal{N}_i^+}) \geq 0,
    \end{equation}
    where $\eta_i, \kappa_i$ are class-$\mathcal{K}$ functions and $\eta_i \in D^r$, with $r \geq 1$.
\end{definition}

In \cite{butler2024collaborativesafety}, we prove the forward invariance of  $\mathcal{C}_i^1 \cap \mathcal{C}_i^2$ given the existence of a CCBF and use this property to construct a decentralized algorithm for \textit{collaborative safety} among cooperating network agents. A high-level pseudo-code description of this procedure is given in Algorithm~\ref{alg:colab_safety}, where, for a given agent $i \in [n]$, safety commitments $\bar{c}_{ij}$ and $\bar{c}_{ki}$ are initialized for incoming neighbors $j \in \mathcal{N}_i^+$ and outgoing neighbors $k \in \mathcal{N}_i^-$, respectively, as well as the set of allowable actions for agent $i$, $\overline{\mathcal{U}}_i$. Then, for each round of collaboration $\tau_i$, agent $i$ computes its maximum safety capability as
\begin{equation}
\label{eq:max_cap}
    \bar{c}_i = \max_{u_i \in \mathcal{U}_i} c_i(\mathbf{x}_i^+, u_i),
\end{equation}
\noindent
which is used as an input for the Collaborate function on Line 5, defined in \cite{butler2024collaborativesafety}, to negotiate safety requests between neighbors $\mathcal{N}_i$ using the collaborative safety condition as defined in \eqref{eq:CCBF_cond}. After each round of collaboration, a safety deficit/surplus $\delta_i$ is returned along with the updated set of allowed actions for agent~$i$ that can satisfy the safety requirements of both agent~$i$ and its neighbors, as well as the current safety commitments for all neighbors. If $\delta_i > 0$, indicating a safety surplus, then Algorithm~\ref{alg:colab_safety} returns the set of feasibly safe actions $\overline{\mathcal{U}}_i$ for agent $i$ and its neighbors $\mathcal{N}_i^-$. In \cite{butler2024collaborativesafety}, we provide conditions under which Algorithm~\ref{alg:colab_safety} will converge asymptotically to at least one safe action for all agents, if it exists. In this paper, we explore the scenario where there is potential error, or non-compliance, in the safety commitments $\bar{c}_{ij}$ given by each neighbor $j \in \mathcal{N}_i^+$ and how resilient each agent is to the collective non-compliance of all of its neighbors. 

\begin{algorithm}
\caption{Collaborative Safety}\label{alg:colab_safety}
    \begin{algorithmic}[1]
        \Initialize{
            $i \gets i_0; \overline{\mathcal{U}}_i \gets \mathcal{U}_i; \tau_i \gets 0$ \\
            $\bar{c}_{ij} \gets 0, \forall j \in \mathcal{N}_i^+; \bar{c}_{ki} \gets 0, \forall k \in \mathcal{N}_i^-$
        }
        \Repeat
            \State $\tau_i \gets \tau_i + 1$
            \State $\bar{c}_i \gets \max_{u_i \in \overline{\mathcal{U}}_i} c_i(\mathbf{x}_i^+, u_i)$
            \State $\delta_i,\, \overline{\mathcal{U}}_i, \{\bar{c}_{ij}\}_{j \in \mathcal{N}_i^+}, \{\bar{c}_{ki}\}_{k \in \mathcal{N}_i^-}$ 
            \Statex \hspace{6ex} \small{$\gets$ Collaborate$\left(\bar{c}_i, \overline{\mathcal{U}}_i, \{\bar{c}_{ij}\}_{j \in \mathcal{N}_i^+}, \{\bar{c}_{ki}\}_{k \in \mathcal{N}_i^-}\right)$ \cite{butler2024collaborativesafety}}
            \normalsize
            \Until{$\delta_i \geq 0$}
        \State\Return $\overline{\mathcal{U}}_i$
    \end{algorithmic}
\end{algorithm}

\vspace{-3ex}

\section{Resilience to Non-Compliance} \label{sec:resiliance}
In this section, we explore properties of resilience in the context of collaborative safety via Algorithm~\ref{alg:colab_safety}. For our discussion, we consider the choice of class-$\mathcal{K}$ functions $\eta_i(\cdot)$ and $\kappa_i(\cdot)$ as a part of $\psi_i^1(\mathbf{x}_i, u_i)$ and $\psi_i^2(\mathbf{x}_i^+, u_i, u_{\mathcal{N}_i^+})$ and their effect on the behavior of neighbors with respect to communicating safety needs to neighbors and choosing actions that guarantee individual node safety. To simplify our discussion, we consider the set of linear class-$\mathcal{K}$ functions.
\begin{assumption}\label{assumne:linear_class_K}
    Let $\eta_i(z) := \eta_i z$ and $\kappa_i(z) := \kappa_i z$, where $z \in \mathbb{R}^{N_i}$ and $\eta_i, \kappa_i \in \mathbb{R}_{\geq 0}$.
\end{assumption}
The selection of $\eta_i$ and $\kappa_i$ may be used to determine the qualitative behavior of a given node, where large $\eta_i$ $\left(\kappa_i\right)$ will prevent the first (second) derivative safety condition in \eqref{eq:MO_funcs} from activating until $x_i$ approaches much closer to the boundary of $\mathcal{C}_i^1$ ($\mathcal{C}_i^2$). In other words, we can use $\eta_i$ and $\kappa_i$ to tune how conservative node~$i$ is in avoiding the boundary of its safety constraints.  
Additionally, we make the following assumption on the control constraints for node $i \in [n]$.
\begin{assumption} \label{assume:U_i_closed}
     Let $\mathcal{U}_i \subset \mathbb{R}^{M_i}$ be a closed set.
\end{assumption}
Note that a key step in collaborative safety via Algorithm~\ref{alg:colab_safety} is the evaluation of the maximum capability of node $i \in [n]$ through \eqref{eq:capability_node_i} and \eqref{eq:max_cap}. Under Assumption~\ref{assumne:linear_class_K} we may express \eqref{eq:capability_node_i} as

\vspace{-2ex}

\footnotesize
\begin{equation}
    \begin{aligned}
        c_i(\mathbf{x}_i^+,u_i,\nu_i) &= \sum_{j \in \mathcal{N}_i^+} \mathcal{L}_{f_j} \mathcal{L}_{f_i} h_i(\mathbf{x}_i^+) +\mathcal{L}^2_{f_i} h_i(\mathbf{x}_i) + u_i^\top \mathcal{L}^2_{g_i} h_i(x_i) u_i\\ 
        & + \kappa_i \eta_i h_i(x_i) + \nu_i \mathcal{L}_{f_i}h_i(\mathbf{x}_i) 
        + \mathcal{L}_{g_i}h_i(x_i)d(u_i) \\ 
        & + \big(\mathcal{L}_{f_i} \mathcal{L}_{g_i}h_i(\mathbf{x}_i)^\top +  \mathcal{L}_{g_i} \mathcal{L}_{f_i}h_i(\mathbf{x}_i) + \nu_i \mathcal{L}_{g_i}h_i(\mathbf{x}_i) \big)u_i,
    \end{aligned}
\end{equation}
\normalsize
where $\nu_i = \eta_i + \kappa_i$. Thus, solving \eqref{eq:max_cap} is equivalent to solving the constrained quadratic program
\small
\begin{equation} \label{eq:max_prob_quad}
    \begin{aligned}
        \max_{u_i \in \overline{\mathcal{U}}_i} \quad & u_i^\top Q_i(x_i) u_i + [q_i(x) + \nu_i \mathcal{L}_{g_i} h_i(x_i)] u_i + \mathcal{L}_{g_i}h_i(x_i)d(u_i),
    \end{aligned}
\end{equation}
\normalsize
with
\begin{equation*}
    Q_i(x_i) = \mathcal{L}^2_{g_i} h_i(x_i)
\end{equation*}

\vspace{-2ex}

\noindent and
\begin{equation*}
    q_i(\mathbf{x}_i) = \mathcal{L}_{f_i} \mathcal{L}_{g_i}h_i(\mathbf{x}_i)^\top +  \mathcal{L}_{g_i} \mathcal{L}_{f_i}h_i(\mathbf{x}_i).
\end{equation*}

Notice, however, that this maximization problem may not necessarily represent the true maximum capability of node~$i$ in satisfying its own safety conditions depending on the properties of $Q_i(x_i)\in \mathbb{R}^{M_i \times M_i}$, i.e., when $Q_i(x_i)$ is negative definite, positive definite, or indefinite. In other words, we must also consider the maximization of $\psi_i^1(\mathbf{x}_i, u_i)$ with respect to its local control
\begin{equation} \label{eq:max_prob_true}
    \begin{aligned}
        \max_{u_i \in \overline{\mathcal{U}}_i} \quad &  \mathcal{L}_{g_i} h_i(x_i) u_i.
    \end{aligned}
\end{equation}
We also consider the following assumption on $\mathcal{L}_{g_i} h_i(x_i)$.
\begin{assumption}\label{assume:non-zero_L_gihi}
    For a given $x_i \in \mathbb{R}^{N_i}$, let all entries of $\mathcal{L}_{g_i} h_i(x_i) \in \mathbb{R}^{M_i}$ be non-zero.
\end{assumption}
In other words, assume that each control input for node~$i$ has some effect on the safety dynamics of $h_i$. Using the above assumptions, we derive the following results.
\begin{lemma} \label{lem:unique_max}
    Given Assumption~\ref{assume:U_i_closed}, $\overline{\mathcal{U}}_i \neq \emptyset$, and $|\mathcal{L}_{g_i} h_i(x_i)| > 0$ for some $x_i \in \mathbb{R}^{N_i}$, \eqref{eq:max_prob_true} must have a maximal solution $u_i^* \in \partial \overline{\mathcal{U}}_i$, with $u_i^* = \argmax_{u_i \in \overline{\mathcal{U}}_i}  \mathcal{L}_{g_i} h_i(x_i) u_i$. Further, if Assumption~\ref{assume:non-zero_L_gihi} holds, then $u_i^*$ is unique.   
\end{lemma}
\begin{proof}
    By \cite[Algorithm 3]{butler2024collaborativesafety}, we have that $\overline{\mathcal{U}}_i$ is computed by taking the intersection of $\mathcal{U}_i$ with  $\overline{\mathcal{U}}_{\mathcal{N}_i^-} = \bigcap_{k \in \mathcal{N}_i^-} \overline{\mathcal{U}}_{ki}$, where
    \begin{equation*}
        \overline{\mathcal{U}}_{ki} = \{ u_i \in \mathbb{R}^{M_i}: a_{ki}(\mathbf{x}_k) u_i + \bar{c}_{ki} + \delta_{ki} \geq 0 \}.
    \end{equation*}
   Thus, since $\overline{\mathcal{U}}_{\mathcal{N}_i^-}$ is the intersection of closed half-spaces, and $\mathcal{U}_i$ is closed by Assumption~\ref{assume:U_i_closed}, we have that $\overline{\mathcal{U}}_i = \mathcal{U}_i \cap \overline{\mathcal{U}}_{\mathcal{N}_i^-}$ must also be closed. 
    Thus, the expression $\mathcal{L}_{g_i} h_i(x_i) u_i$ evaluated at some $x_i \in \mathbb{R}^{N_i}$ describes a closed hyperplane in $\mathbb{R}^{M_i}$ with at least one non-zero gradient. Therefore, if we define $u_i^* = [u_i^{1*}, \dots u_i^{M_i*}]^\top$, $\mathcal{L}_{g_i} h_i(x_i) = [\mathcal{L}_{g_i} h_i(x_i)^1, \dots, \mathcal{L}_{g_i} h_i(x_i)^{M_i}]$, and $\overline{\mathcal{U}}_i = \overline{\mathcal{U}}_i^1 \times \dots \times \overline{\mathcal{U}}_i^{M_i} \neq \emptyset$, then, for each $m \in [M_i]$,
    \begin{equation} \label{eq:u_i_m-star}
        u_i^{m*} = \argmax_{u_i^{m} \in \partial \overline{\mathcal{U}}_i^{m}} \mathcal{L}_{g_i} h_i(x_i)^m u_i^m.
    \end{equation}
    Further, if $|\mathcal{L}_{g_i} h_i(x_i)^m| > 0$ for all $m \in [M_i]$, i.e., if Assumption~\ref{assume:non-zero_L_gihi} holds, then $u_i^{m*}$ is unique for $m \in [M_i]$ since \eqref{eq:u_i_m-star} is maximizing linear variables with non-zero gradients constrained by the closed set $\overline{\mathcal{U}}_i \subset \mathbb{R}^{M_i}$.
\end{proof}
Since $\mathcal{L}_{g_i} h_i(x_i) u_i$ is a term in \eqref{eq:max_prob_quad}, we gain the following insight on the relationship between \eqref{eq:max_prob_quad} and \eqref{eq:max_prob_true}.

\begin{lemma} \label{lem:nu_star}
    Let Assumptions~\ref{assume:u_dot_func_u}-\ref{assume:non-zero_L_gihi} hold for some $x_i \in \mathbb{R}^{N_i}$ and $x_{\mathcal{N}_i^+} \in \mathbb{R}^{N_i^+}$. Further, let
    \begin{equation} \label{eq:argmax_true}
        u_i^{*} := \argmax_{u_i \in \overline{\mathcal{U}}_i} \mathcal{L}_{g_i} h_i(x_i) u_i
    \end{equation}
    and
    \begin{equation} \label{eq:argmax_quad}
        \begin{aligned}
            u_i^{c} := & \argmax_{u_i \in \overline{\mathcal{U}}_i} & u_i^\top Q_i(x_i) u_i + \mathcal{L}_{g_i}h_i(x_i)d(u_i) 
            \\ & & + \left[q_i(\mathbf{x}_i) + \nu_i \mathcal{L}_{g_i} h_i(x_i) \right] u_i.
        \end{aligned}
    \end{equation}
    % there exists a $\nu_i^* > 0$ such that for all $\nu_i \geq \nu_i^*$ 
    Then, as $\nu_i \rightarrow \infty$, we have that $u_i^* = u_i^{c}$. 
    % Further, we have that
    % \begin{equation}
    %     \mathcal{L}_{g_i}h_i(x_i)(u_i^{*} - u_i^{c}) \geq 0.
    % \end{equation}
\end{lemma}
\begin{proof}
     Note that in \eqref{eq:argmax_quad} as we increase $\nu_i$ the linear term $\mathcal{L}_{g_i} h_i(x_i)$ will begin to dominate the expression for the capability of node $i \in [n]$. Therefore, by Assumptions~\ref{assumne:linear_class_K} and \ref{assume:U_i_closed} we have
    \footnotesize
    \begin{equation*}
        \begin{aligned}
            \lim_{\nu_i \rightarrow \infty} & \frac{1}{\nu_i}\left(u_i^\top Q_i(x_i) u_i + q_i(\mathbf{x}_i)u_i + \mathcal{L}_{g_i}h_i(x_i)d(u_i) \right) + \mathcal{L}_{g_i} h_i(x_i) u_i \\ 
            &=  \mathcal{L}_{g_i} h_i(x_i) u_i.
        \end{aligned}
    \end{equation*}
    \normalsize
    Thus, as $\nu_i \rightarrow \infty$ we have $u_i^* = u_i^{c}$, where $u_i^*$ is unique by Lemma~\ref{lem:unique_max}.   
\end{proof}

Given these results, we may quantify the discrepancy between the communicated capability and the actual capability of node~$i$ with respect to a chosen $\nu_i$. 
We define the quantity of neighbor compliance as follows. First, define the minimum assistance requested from neighbor $j \in \mathcal{N}_i^+$ by node~$i$ as the set of control actions
\begin{equation}
    U_j^m(x) = \{u_j \in \overline{\mathcal{U}}_j: a_{ij}(\mathbf{x}_i) u_j + \bar{c}_{ij} = 0 \}, 
\end{equation}
where $\bar{c}_{ij}$ is the safety commitment allocated to node $j$ by node~$i$ according to Algorithm~\ref{alg:colab_safety}. We compute the amount of \textit{compliance} by neighbor $j$ to requests by node~$i$ given any action $u_j \in \mathcal{U}_j$ as
\begin{equation} \label{eq:compliance}
    e_{ij}(\mathbf{x}_i) := a_{ij}(\mathbf{x}_i)(u_j^m - u_j),
\end{equation}
where we may choose any $u_j^m \in U_j^m(\mathbf{x}_i)$ without loss of generality. 
\begin{definition}
    If $e_{ij}(\mathbf{x}_i) < 0$, then neighbor $j$ is \textbf{non-compliant}. The \textbf{neighborhood compliance} of incoming neighbors to node~$i$ is given by $e_i(\mathbf{x}_i) = \sum_{j \in \mathcal{N}_i^+} e_{ij}(x)$, where if $e_i(\mathbf{x}_i) < 0$, then the incoming neighborhood of node~$i$ is \textbf{non-compliant}. 
\end{definition}
If we define $\nu_i^*$ as the minimum $\nu_i$ such that $u_i^c = u_i^*$, then we may compute the upper bound on the total non-compliance agent~$i$ is resilient to as follows.
\begin{theorem}\label{thm:neighbor_error_tol}
    For a given $x \in \mathbb{R}^{N}$ let $\nu_i \leq \nu_i^*$, $u_i^*$, and $u_i^c$ be computed via \eqref{eq:argmax_true} and \eqref{eq:argmax_quad}, respectively. If Assumptions~\ref{assume:u_dot_func_u}-\ref{assume:non-zero_L_gihi} hold and $e_i(\mathbf{x}_i) < 0$,
    then node~$i$ will be capable of satisfying its safety constraints in the presence of non-compliance from its incoming neighborhood up to
    \begin{equation} \label{eq:error_upper_bound}
        \begin{aligned}
            |e_i(\mathbf{x}_i)| & \leq (u_i^*-u_i^c)^\top Q_i(x_i)(u_i^*-u_i^c) + q_i(\mathbf{x}_i)(u_i^*-u_i^c) \\
            & \quad + \mathcal{L}_{g_i}h_i(x_i)(d(u_i^*) - d(u_i^c)) \\
            & \quad + \nu_i^* \mathcal{L}_{g_i}h_i(x_i)u_{i}^{*} - \nu_i \mathcal{L}_{g_i}h_i(x_i)u_{i}^{c}.
        \end{aligned}
    \end{equation}
\end{theorem}
\begin{proof}
    By Assumption~\ref{assume:non-zero_L_gihi} and Lemma~\ref{lem:unique_max} we that $u_i^* \in \partial \overline{\mathcal{U}}_i$ is unique and therefore $\mathcal{L}_{g_i} h_i(x_i) u_i^* >  \mathcal{L}_{g_i} h_i(x_i) u_i, \forall u_i \in \overline{\mathcal{U}}_i \setminus \{u_i^*\}$. Thus, for all $\nu_i \leq \nu_i^*$, we have
    \begin{equation}\label{eq:greater_capability}
        c_i(\mathbf{x}_i^+,u_i^*,\nu_i^*) \geq c_i(\mathbf{x}_i^+,u_i^c,\nu_i) 
    \end{equation}
    for a given $x \in \mathbb{R}^N$. Consider when
    \begin{equation*}
        c_i(\mathbf{x}_i^+,u_i^*,\nu_i^*) \leq 0
    \end{equation*}
    and there exists $\nu_i > 0$ such that 
    \begin{equation} \label{eq:equally_satisfied}
         \psi_i^2(\mathbf{x}_i^+,u_i^c,u^c_{\mathcal{N}_i^+}, \nu_i) = \psi_i^2(\mathbf{x}_i^+,u_i^*,u^*_{\mathcal{N}_i^+}, \nu_i^*).
    \end{equation}
    Thus, by \eqref{eq:equally_satisfied}, we have
    \begin{equation}
        \begin{aligned}
            \sum_{j \in \mathcal{N}_i^+} a_{ij}(u_j^* - u_j^c) + c_i(\mathbf{x}_i^+,u_i^*,\nu_i^*) - c_i(\mathbf{x}_i^+,u_i^c,\nu_i) = 0
        \end{aligned}
    \end{equation}
    and by \eqref{eq:greater_capability} we have
    \begin{equation*}
        \sum_{j \in \mathcal{N}_i^+} a_{ij}(u_j^* - u_j^c) \leq 0.
    \end{equation*}
    In this way, we characterize the difference between the true maximum capability of node~$i$ and its communicated maximum capability as
    \begin{equation} \label{eq:neighbor_error_tol}
        \begin{aligned}
            \epsilon_i(\mathbf{x}_i, \nu_i) &:= (u_i^* - u_i^c)^{\top} Q_i(x_i)(u_i^* - u_i^c) + q_i(\mathbf{x}_i)(u_i^* - u_i^c) \\
             & \quad + \mathcal{L}_{g_i}h_i(x_i)(d(u_i^*) - d(u_i^c)) \\
            & \quad + \nu_i^* \mathcal{L}_{g_i}h_i(x_i)u_{i}^{*} - \nu_i \mathcal{L}_{g_i}h_i(x_i)u_{i}^{c} \geq 0,
        \end{aligned} 
    \end{equation}
    where, by Assumptions~\ref{assumne:linear_class_K}-\ref{assume:non-zero_L_gihi} and Lemma~\ref{lem:nu_star}, we have if $\nu_i \geq \nu_i^*$, then $\epsilon_i(\mathbf{x}_i, \nu_i) = 0$. Therefore, if we consider any deficiency between the minimum requested assistance via Algorithm~\ref{alg:colab_safety} and the actual actions of neighbors $e_{ij} = a_{ij}(u_j^m - u_j)$ for $j \in \mathcal{N}_i^+$, if $\left| \sum_{j \in \mathcal{N}_i^+} e_{ij} \right| \leq \epsilon_i(\mathbf{x}_i, \nu_i)$ then there exists a $u_i^* \in \overline{\mathcal{U}}_i$ and $\nu_i^*$ such that $\sum_{j \in \mathcal{N}_i^+} a_{ij}(\mathbf{x}_i) u_j + c_i(\mathbf{x}_i^+, u_i^*,\nu_i^*) \geq 0$.   
\end{proof}

Note that $u_i^*$ is a function of $x_i$ and $u_i^c$ is a function of $\mathbf{x}_i$ and $\nu_i$. Thus, by Theorem~\ref{thm:neighbor_error_tol} we may compute the resilience that node~$i$ has to neighborhood non-compliance with respect to a given neighborhood state and choice of $\nu_i = \eta_i + \kappa_i$. Further, note that if $\nu_i \geq \nu_i^*$, then $u_i^c=u_i^*$ and \eqref{eq:error_upper_bound} becomes $|e_i| = 0$. Thus, we can use $\nu_i^*$ to characterize a kind of \textit{resilience boundary} for a given 1-hop neighborhood state $\mathbf{x}_i$, where any $\nu_i \geq \nu_i^*$ characterizes a system with no resilience.

This measure of resilience may be useful in considering parameter selection for each node in the network depending on the reliability of neighboring nodes and the consequences of node failure. However, it should also be noted that, as \eqref{eq:neighbor_error_tol} increases, node~$i$ will place a greater burden of its safety on its neighbors $j \in \mathcal{N}_i^+$, which could lead to an over-taxing of neighbor control resources and under-utilization of the resources of node~$i$. Therefore, the context and consequences of over-requesting by neighbors in specific model applications should be considered when selecting $\eta_i$ and $\kappa_i$. We illustrate these behaviors on networked epidemic processes in the following section.

\section{Simulations} \label{sec:simulations}

We define an SIS networked epidemic spreading process \cite{pare2020modeling} with $n$ nodes where $x_i$ is the proportion of infected population at node $i\in [n]$, with the infection dynamics defined by
\vspace{-2ex}

\begin{equation} \label{eq:sis_networked}
    \dot{x}_i = -(\gamma_i+u_i) x_i + (1-x_i)\sum_{j \in [n]}\beta_{ij} x_j,
\end{equation}

\noindent where $\gamma_i>0$ is the recovery rate at node~$i$, $u_i \in \mathcal{U}_i \subset \mathbb{R}_{\geq 0}$ is a control input boosting the healing rate at node~$i$, and $\beta_{ij} \geq 0$ is the networked connection going from node $j$ to node~$i$ (note $\beta_{ii}$ is simply the infection rate occurring at node~$i$).

 Let $\mathcal{U}_i = [0,\bar{u}_i]$, where $\bar{u}_i$ is the upper limit on boosting the healing rate at node~$i$. We let each node define its individual safety constraint as
\begin{equation} \label{eq:sis_indv_bar_func}
    h_i(x_i) = \bar{x}_i - x_i,
\end{equation}
where $\bar{x}_i \in (0,1]$ is the defined safety threshold for the acceptable proportion of infected individuals at node~$i$. Thus, the individual safe sets for each node are given by
\vspace{-1.5ex}

\begin{equation}\label{eq:sis_indv_safe_set}
    \mathcal{C}_i = \{ x \in [0,1]^n: h_i(x_i) \geq 0 \}.
\end{equation}

For this example, we construct a simple 3-node system where $\beta_{ij} = 0.25$ and $\beta_{ii} = 0.5$, for all $j \neq i$ and $j = i$, respectively. To induce an endemic state, we set the healing rate $\gamma_i = 0.3$, for all nodes. We set the safety constraints for each node as $\bar{x}_1 = 0.1$, $\bar{x}_2 = 0.12$, and $\bar{x}_3 = 0.18$ and set the initial infection levels at $x = [0.04, 0.01, 0.02]$.

To illustrate how $\nu_i = \eta_i + \kappa_i$ can be used to increase resilience to sudden changes or system failures, we simulate the same system for both large $\nu_i$ ($\eta_i = 10$ and $\kappa_i = 1$) and small $\nu_i$ ($\eta_i = 0.3$ and $\kappa_i = 0.3$), for all $i \in [3]$, in Figures~\ref{fig:eta10_kappa1} and \ref{fig:eta_kappa_p3}, respectively. In both simulations, we set the control limits $\mathcal{U}_i(t) = [0,0.75]$ for $t < 10$ and then perturb the system by suddenly reducing the control ability to $\mathcal{U}_i(t) = [0, 0.6]$ for $t \geq 10$, for all $i \in[3]$. Note that the collaborative safety framework maintains safety for both instances up to $t=10$, and when the system experiences the failure at $t=10$ the dynamics react accordingly. For the less resilient system, shown in Figure~\ref{fig:eta10_kappa1}, we see node~$1$ violates its safety requirement after the system failure at $t=10$. By contrast, we see in Figure~\ref{fig:eta_kappa_p3} that when $\nu_i$ is small for all $i \in [3]$, it causes node~$1$ to over-request assistance from its neighbors, leaving increased control ability for node~$1$ that allows it to be resilient to the network failure, consistent with the result in Theorem~\ref{thm:neighbor_error_tol}. Thus we see that the network is able to operate reliably even after experiencing some failures when $\nu_i$ is sufficiently small for all $i$, i.e., when $\nu_i \ll \nu_i^*$. 

Recall from Theorem~\ref{thm:neighbor_error_tol} that $\nu_i^*$ is the minimum $\nu_i$ such that $u_i^c = u_i^*$, where $u_i^*$ and $u_i^c$ are computed via \eqref{eq:argmax_true} and \eqref{eq:argmax_quad}, respectively.
To visualize the resilience boundary, $\nu_i^*$, of node 1 for a given network state, we fix $x_2 = x_3 = 0.1$ and approximate $\nu_1^*$ in Figure~\ref{fig:nu_star} by incrementing $\nu_1$ by $\delta_{\nu}=0.01$ until $u_1^* = u_1^c$ for $x_1 = [0, \bar{x}_1]$. In Figure~\ref{fig:nu_star}, we see that, as $x_1$ approaches $\bar{x}_1$, the resilience boundary increases to include larger values of $\nu_1$ that would enable resilience for node~$1$ to non-compliant neighbors. We then use this approximation of $\nu_i^*$
to generate the surface in Figure~\ref{fig:neighbor_error_tol_surf} by computing $\epsilon_i(x, \nu_i)$ using \eqref{eq:neighbor_error_tol} for $x_i \in [0, \bar{x}_1]$ and $\nu_1 \in [0, 0.8]$. Note that for small $x_1$ we have that the linear term $q_1(x) =  \mathcal{L}_{g_i}\mathcal{L}_{f_i}h_i(\mathbf{x}_i)^\top + \mathcal{L}_{g_i}\mathcal{L}_{f_i}h_i(\mathbf{x}_i)$ overpowers the negative definite quadratic term $Q_1(x_1)$, thereby setting $\nu_1^* = 0$. 
Therefore, we see that the tolerance for neighborhood non-compliance grows as $\nu_1$ decreases, and as $x_1$ approaches the barrier of its safety constraints it will require a larger amount of error tolerance from its neighbors.

\begin{figure}
    \centering
    \captionsetup{belowskip=-12pt}
    \begin{subfigure}[b]{0.36\textwidth}  
        \centering 
        \includegraphics[width=\textwidth]{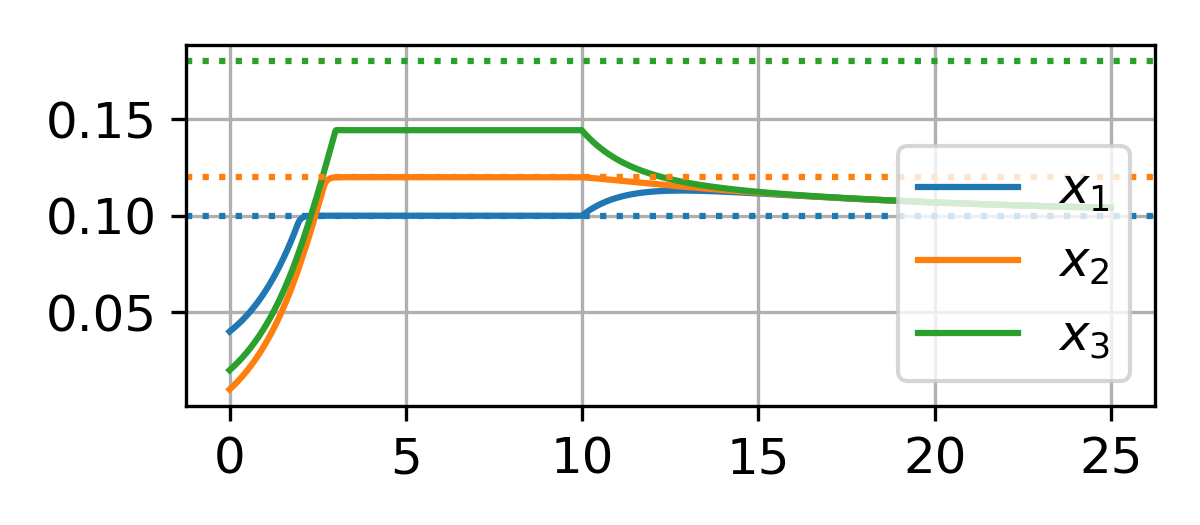}
    \end{subfigure}
    \begin{subfigure}[b]{0.36\textwidth}   
        \centering 
        \includegraphics[width=\textwidth]{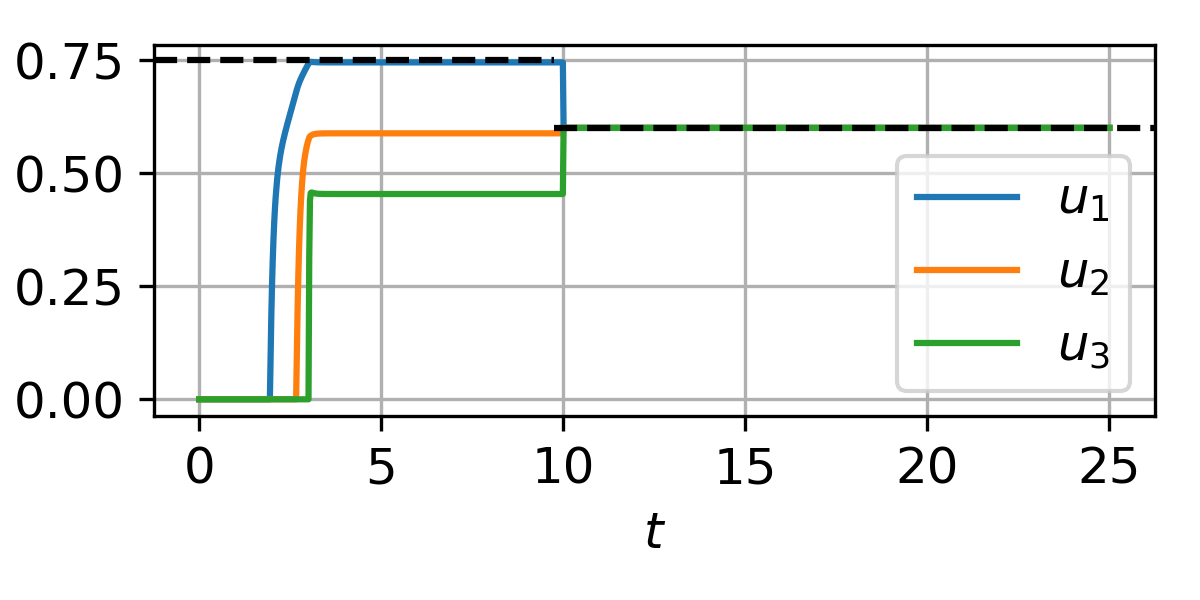}
        % \caption{{\small $\eta_i = 10$, $\kappa_i = 1$}}    
        % \label{fig:eta_10_kappa_1_resilient}
    \end{subfigure}
    \caption{
    Simulated networked SIS model from \eqref{eq:sis_networked}, with each node implementing the collaborative safety framework from \cite{butler2024collaborativesafety}, where $n=3$ and $\eta_i=10, \kappa_i=1$ for all $i \in [3]$. Each node's safety constraint $\bar{x}_i$ is represented by the dotted line of corresponding color and the input constraint $\mathcal{U}_i(t)$ for all nodes $i \in [3]$ is shown by the black dotted line, where $\mathcal{U}_i(t) = [0,0.75]$ if $t < 10$ and $\mathcal{U}_i(t) = [0,0.6]$ if $t \geq 10$. Note that the system failure at $t=10$ causes node~$1$ to violate its safety constraint.
    } 
    \label{fig:eta10_kappa1}
\end{figure}

\begin{figure}
    \centering
    \captionsetup{belowskip=-10pt, aboveskip=-5pt}
    \begin{subfigure}[b]{0.36\textwidth}  
        \centering 
        \includegraphics[width=\textwidth]{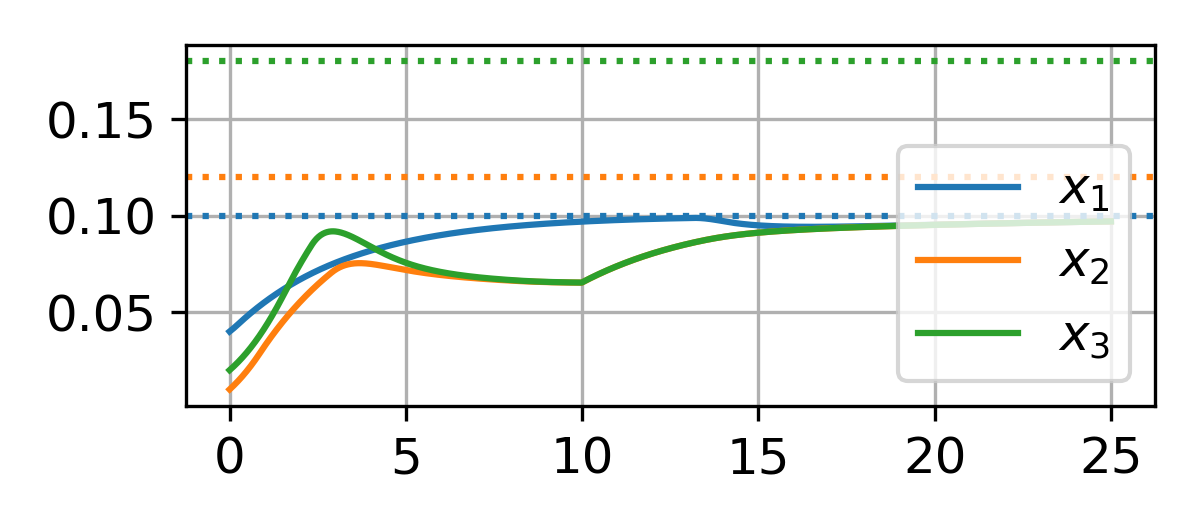}
    \end{subfigure}
    \begin{subfigure}[b]{0.36\textwidth}  
        \centering 
        \includegraphics[width=\textwidth]{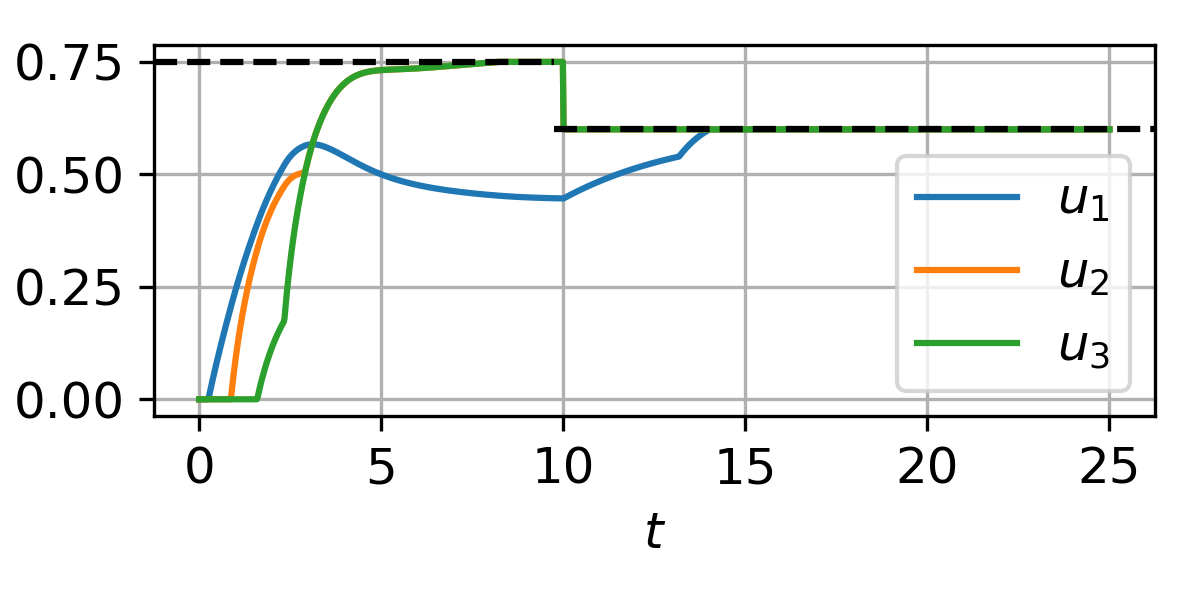}
        % \caption{{\small $\eta_i = 0.3$, $\kappa_i = 0.3$}}    
        % \label{fig:eta_p3_kappa_p3_resilient}
    \end{subfigure}
    \caption{
    The same simulation from Figure~\ref{fig:eta10_kappa1} except $\eta_i=\kappa_i=0.3$ for all $i \in [3]$. Note that, unlike in Figure~\ref{fig:eta10_kappa1}, node~$1$ is able to satisfy its safety constraint for all $t$. \vspace{-1ex}
    } 
    \label{fig:eta_kappa_p3}
\end{figure}

\begin{figure}
    \centering
    \captionsetup{belowskip=-16pt, aboveskip=-5pt}
    \includegraphics[width=.6\columnwidth,trim={.5cm 0 0 0},clip]{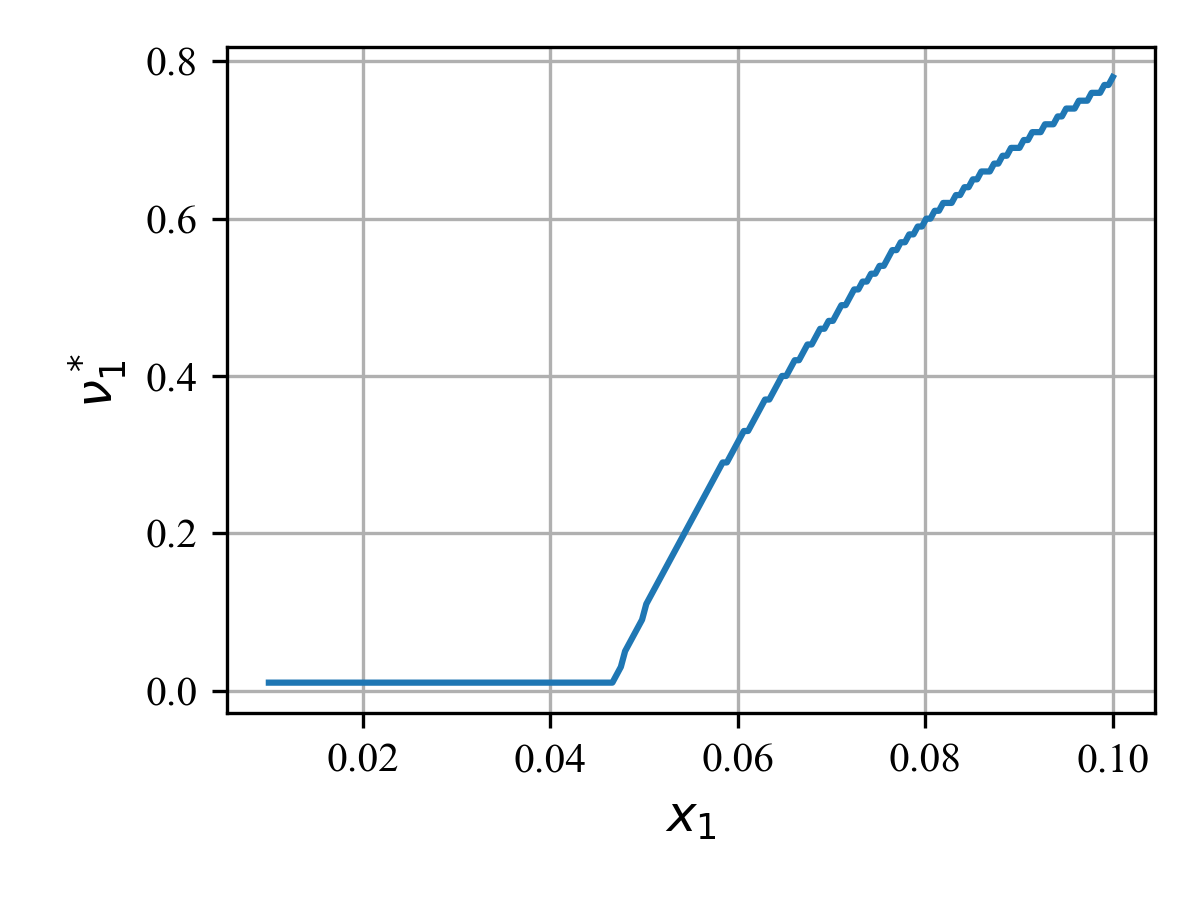}
    \caption{A numerical approximation of $\nu_1^*$ for $x_2 = x_3 = 0.1$ with $x_1 \in [0, \bar{x}_1]$, where $\nu_1$ is incremented by $\delta_{\nu} = 0.01$ until $u_1^* = u_1^c$. Notice as $x_1$ approaches $\bar{x}_1$, the resilience boundary increases to include larger values of $\nu_1$ that would enable resilience for node~$1$ to non-compliant neighbors. \vspace{-1ex}}
    \label{fig:nu_star}
\end{figure}

\begin{figure}
    \centering
    \captionsetup{belowskip=-15pt}
    \begin{subfigure}[b]{0.7\columnwidth}
        \centering 
        \includegraphics[width=\textwidth]{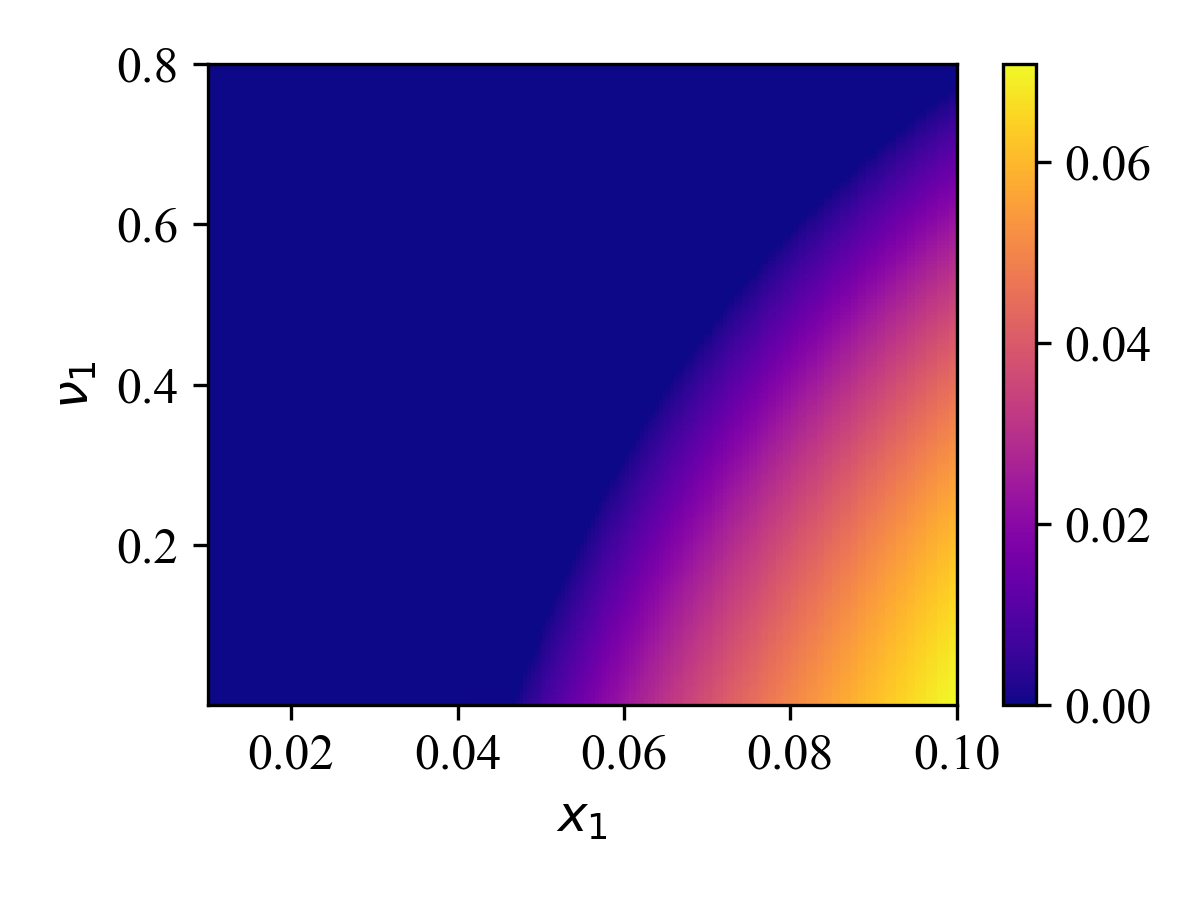}
    \end{subfigure}
    \caption{
    The error tolerance in neighbor request fulfillment $\epsilon_i(x, \nu_i)$, defined by \eqref{eq:neighbor_error_tol}, computed for $x_2 = x_3 = 0.1$ with $x_1 \in [0, \bar{x}_1]$ and $\nu_1 \in [0, 0.8]$. Notice as $x_1$ approaches the barrier of its safety constraints it requires a larger amount of error tolerance from its neighbors.
    }
    \label{fig:neighbor_error_tol_surf}
\end{figure}

\section{Conclusion} \label{sec:conclusion}
In this letter, we have explored properties of resilience for coupled agents with respect to non-compliance to safety requests made to neighbors under a collaborative safety framework. We have computed an upper bound on the total non-compliance each agent may tolerate while still being able to satisfy their individual safety needs given its 1-hop neighborhood state and knowledge of networked dynamics. Future work includes extending this property of non-compliance to include general uncertainty and noise in the coupled dynamic system. Additionally, thus far our collaborative framework assumes all agents are cooperative; however, since the notion of non-compliance naturally implies scenarios where some agents may act maliciously, further work will account for such non-cooperative agents. 

\normalem
\bibliographystyle{IEEEtran}
\bibliography{references}

\begin{thebibliography}{10}
\providecommand{\url}[1]{#1}
\csname url@samestyle\endcsname
\providecommand{\newblock}{\relax}
\providecommand{\bibinfo}[2]{#2}
\providecommand{\BIBentrySTDinterwordspacing}{\spaceskip=0pt\relax}
\providecommand{\BIBentryALTinterwordstretchfactor}{4}
\providecommand{\BIBentryALTinterwordspacing}{\spaceskip=\fontdimen2\font plus
\BIBentryALTinterwordstretchfactor\fontdimen3\font minus \fontdimen4\font\relax}
\providecommand{\BIBforeignlanguage}[2]{{%
\expandafter\ifx\csname l@#1\endcsname\relax
\typeout{** WARNING: IEEEtran.bst: No hyphenation pattern has been}%
\typeout{** loaded for the language `#1'. Using the pattern for}%
\typeout{** the default language instead.}%
\else
\language=\csname l@#1\endcsname
\fi
#2}}
\providecommand{\BIBdecl}{\relax}
\BIBdecl

\bibitem{tuballa2016review}
M.~L. Tuballa and M.~L. Abundo, ``A review of the development of smart grid technologies,'' \emph{Renewable and Sustainable Energy Reviews}, vol.~59, pp. 710--725, 2016.

\bibitem{tahir2019swarms}
A.~Tahir, J.~B{\"o}ling, M.-H. Haghbayan, H.~T. Toivonen, and J.~Plosila, ``Swarms of unmanned aerial vehicles—{A} survey,'' \emph{Journal of Industrial Information Integration}, vol.~16, p. 100106, 2019.

\bibitem{chung2018survey}
S.-J. Chung, A.~A. Paranjape, P.~Dames, S.~Shen, and V.~Kumar, ``A survey on aerial swarm robotics,'' \emph{IEEE Transactions on Robotics}, vol.~34, no.~4, pp. 837--855, 2018.

\bibitem{plathottam2018next}
S.~J. Plathottam and P.~Ranganathan, ``Next generation distributed and networked autonomous vehicles,'' in \emph{Proceedings of the 2018 10th International Conference on Communication Systems \& Networks (COMSNETS)}.\hskip 1em plus 0.5em minus 0.4em\relax IEEE, 2018, pp. 577--582.

\bibitem{wang2017safety}
L.~Wang, A.~D. Ames, and M.~Egerstedt, ``Safety barrier certificates for collisions-free multirobot systems,'' \emph{IEEE Transactions on Robotics}, vol.~33, no.~3, pp. 661--674, 2017.

\bibitem{pare2020modeling}
P.~E. Par{\'e}, C.~L. Beck, and T.~Ba{\c{s}}ar, ``Modeling, estimation, and analysis of epidemics over networks: An overview,'' \emph{Annual Reviews in Control}, vol.~50, pp. 345--360, 2020.

\bibitem{blanchini1999set}
F.~Blanchini, ``Set invariance in control,'' \emph{Automatica}, vol.~35, no.~11, pp. 1747--1767, 1999.

\bibitem{ames2016control}
A.~D. Ames, X.~Xu, J.~W. Grizzle, and P.~Tabuada, ``Control barrier function based quadratic programs for safety critical systems,'' \emph{IEEE Transactions on Automatic Control}, vol.~62, no.~8, pp. 3861--3876, 2016.

\bibitem{lewis2013cooperative}
F.~L. Lewis, H.~Zhang, K.~Hengster-Movric, and A.~Das, \emph{Cooperative Control of Multi-agent Systems: Optimal and Adaptive Design Approaches}.\hskip 1em plus 0.5em minus 0.4em\relax Springer Science \& Business Media, 2013.

\bibitem{li2017cooperative}
Z.~Li and Z.~Duan, \emph{Cooperative Control of Multi-agent Systems: A Consensus Region Approach}.\hskip 1em plus 0.5em minus 0.4em\relax CRC Press, 2017.

\bibitem{wang2017cooperative}
Y.~Wang, E.~Garcia, D.~Casbeer, and F.~Zhang, \emph{Cooperative Control of Multi-agent Systems: Theory and Applications}.\hskip 1em plus 0.5em minus 0.4em\relax John Wiley \& Sons, 2017.

\bibitem{yu2017distributed}
W.~Yu, G.~Wen, G.~Chen, and J.~Cao, \emph{Distributed Cooperative Control of Multi-agent Systems}.\hskip 1em plus 0.5em minus 0.4em\relax John Wiley \& Sons, 2017.

\bibitem{ramachandran2021resilient}
R.~K. Ramachandran, P.~Pierpaoli, M.~Egerstedt, and G.~S. Sukhatme, ``Resilient monitoring in heterogeneous multi-robot systems through network reconfiguration,'' \emph{IEEE Transactions on Robotics}, vol.~38, no.~1, pp. 126--138, 2021.

\bibitem{dobson2019self}
S.~Dobson, D.~Hutchison, A.~Mauthe, A.~Schaeffer-Filho, P.~Smith, and J.~P. Sterbenz, ``Self-organization and resilience for networked systems: Design principles and open research issues,'' \emph{Proceedings of the IEEE}, vol. 107, no.~4, pp. 819--834, 2019.

\bibitem{filippini2014modeling}
R.~Filippini and A.~Silva, ``A modeling framework for the resilience analysis of networked systems-of-systems based on functional dependencies,'' \emph{Reliability Engineering \& System Safety}, vol. 125, pp. 82--91, 2014.

\bibitem{reed2009methodology}
D.~A. Reed, K.~C. Kapur, and R.~D. Christie, ``Methodology for assessing the resilience of networked infrastructure,'' \emph{IEEE Systems Journal}, vol.~3, no.~2, pp. 174--180, 2009.

\bibitem{butler2024collaborativesafety}
\BIBentryALTinterwordspacing
B.~A. Butler and P.~E. Par\'{e}, ``Collaborative safety-critical control for dynamically coupled networked systems,'' 2024. [Online]. Available: \url{https://arxiv.org/abs/2310.03289}
\BIBentrySTDinterwordspacing

\bibitem{butler2024collabformation}
B.~A. Butler, C.~H. Leung, and P.~E. Par{\'e}, ``Collaborative safe formation control for coupled multi-agent systems,'' in \emph{Proceedings of the 2024 European Control Conference (ECC)}.\hskip 1em plus 0.5em minus 0.4em\relax IEEE, 2024, pp. 3410--3415.

\bibitem{valdez2020cascading}
L.~D. Valdez, L.~Shekhtman, C.~E. La~Rocca, X.~Zhang, S.~V. Buldyrev, P.~A. Trunfio, L.~A. Braunstein, and S.~Havlin, ``Cascading failures in complex networks,'' \emph{Journal of Complex Networks}, vol.~8, no.~2, p. cnaa013, 2020.

\bibitem{robenack2008computation}
K.~R{\"o}benack, ``Computation of multiple {L}ie derivatives by algorithmic differentiation,'' \emph{Journal of Computational and Applied Mathematics}, vol. 213, no.~2, pp. 454--464, 2008.

\bibitem{xiao2021high}
W.~Xiao and C.~Belta, ``High-order control barrier functions,'' \emph{IEEE Transactions on Automatic Control}, vol.~67, no.~7, pp. 3655--3662, 2021.

\end{thebibliography}

\end{document}